\documentclass[preprint,12pt]{elsarticle}

\usepackage{amsmath}
\usepackage{a4wide}
\usepackage[utf8]{inputenc}
\usepackage{amssymb}
\usepackage{amsopn}
\usepackage{epsfig}
\usepackage{amsfonts}
\usepackage{latexsym}
\usepackage{amsthm}
\usepackage{enumerate}
\usepackage[UKenglish]{babel}
\usepackage{verbatim}
\usepackage{color}
\usepackage{pgf}
\usepackage{tikz}

\DeclareMathAlphabet{\pazocal}{OMS}{zplm}{m}{n}

\newtheorem{theorem}{Theorem}[section]
\newtheorem{lemma}[theorem]{Lemma}
\newtheorem{proposition}[theorem]{Proposition}

\theoremstyle{definition}
\newtheorem{definition}[theorem]{Definition}

\newtheorem{conjecture}[theorem]{Conjecture}

\theoremstyle{remark}
\newtheorem{remark}[theorem]{Remark}
\newtheorem{remarks}[theorem]{Remarks}

\numberwithin{equation}{section}

\newcommand{\R}{\ensuremath{\mathbb{R}}}
\newcommand{\N}{\ensuremath{\mathbb{N}}}

\renewcommand{\u}{\ensuremath{\pazocal{U}}}

\newcommand{\us}{\mathbf{U}}
\newcommand{\n}{\ensuremath{\pazocal{N}}}
 
\newcommand{\I}{\mathbf{I}}
\newcommand{\J}{\mathbf{J}}
\newcommand{\G}{\mathbf{G}}

\newcommand{\ns}{{\mathbf{N}}}

\newcommand{\set}[1]{\left\{#1\right\}}
\newcommand{\la}{\lambda}

\newcommand{\ep}{\varepsilon}
\newcommand{\f}{\infty}
\newcommand{\de}{\delta}
\newcommand{\om}{\omega}
\newcommand{\al}{\alpha}
\newcommand{\lle}{\preccurlyeq}
\newcommand{\lge}{\succcurlyeq}

\newcommand{\si}{\sigma}

\newcommand{\ra}{\rightarrow}

\newcommand{\qtq}[1]{\quad\text{#1}\quad}

\begin{document}
\begin{frontmatter}

\title{Algebraic sums and products of univoque bases}

 \author[K. Dajani]{Karma Dajani}
\address[K. Dajani]{Department of Mathematics, Utrecht University, Fac Wiskunde en informatica and MRI, Budapestlaan 6, P.O. Box 80.000, 3508 TA Utrecht, The Netherlands}
\ead{k.dajani1@uu.nl}

\author[V. Komornik]{Vilmos Komornik}
\address[V. Komornik]{D\'{e}partement de math\'{e}matique,
         Universit\'{e} de Strasbourg,
         7 rue Ren\'{e} Descartes,
         67084 Strasbourg Cedex, France}
\ead{komornik@math.unistra.fr}

\author[D. Kong]{Derong Kong\corref{kong}}
\address[D. Kong]{Mathematical Institute, University of Leiden, PO Box 9512, 2300 RA Leiden, The Netherlands}
\cortext[kong]{Corresponding author}\ead{derongkong@126.com}

\author[W. Li]{Wenxia Li}
\address[W. Li]{Department of Mathematics, Shanghai Key Laboratory of PMMP, East China Normal University, Shanghai 200062,
People's Republic of China}
\ead{wxli@math.ecnu.edu.cn}


%

\begin{abstract}
Given $x\in(0, 1]$, let $\u(x)$ be the set of bases $q\in(1,2]$ for which there exists a unique sequence $(d_i)$ of zeros and ones such that $x=\sum_{i=1}^\f{{d_i}/{q^i}}$.
L\"{u}, Tan and Wu \cite{Lu_Tan_Wu_2014} proved that $\u(x)$ is a Lebesgue null set of full Hausdorff dimension. 
In this paper, we show that the algebraic sum $\u(x)+\la\u(x)$ {and  product $\u(x)\cdot\u(x)^\la$}  contain an interval for all $x\in(0, 1]$ and $\la\ne 0$.   
As an application we show that the same phenomenon occurs for the set of non-matching parameters studied by the first author and Kalle \cite{Daj-Kal-07}.
\end{abstract}
\end{frontmatter}

\section{Introduction}\label{s1}
Non-integer base expansions, a natural extension of dyadic expansions, have got much  attention since the ground-breaking works of R\'enyi \cite{Renyi_1957} and Parry \cite{Parry_1960}. 
Given a base $q\in(1,2]$, an infinite sequence $(d_i)$ of zeros and ones is called a \emph{$q$-expansion} of $x$ if 
\begin{equation*}
x=\sum_{i=1}^\f\frac{d_i}{q^i}=:((d_i))_q.
\end{equation*}
A number $x$ has a $q$-expansion if and only if $x\in I_q:=[0, \frac{1}{q-1}]$. 
Contrary to the the dyadic expansions, Lebesgue almost every $x\in I_q$ has a continuum of $q$-expansions (see \cite{Sidorov_2003}). 
On the other hand, for each $k\in\N:=\set{1,2,\ldots}$ or $k=\aleph_0$ there exist $q\in(1, 2]$ and $x\in I_q$ such that $x$ has precisely $k$ different $q$-expansions (see \cite{Erdos_Joo_1992}). 
For more information on the non-integer base expansions we refer to the survey paper \cite{Komornik_2011} and the book chapter \cite{deVries-Komornik-2016}.

On the other hand, algebraic differences of Cantor sets and their connections with dynamical systems have been intensively investigated since the  work of  Newhouse \cite{Newhouse-1979}, who  introduced the notion of \emph{thickness} to study whether a given Cantor set $C\subset\R$ has a non-empty intersection with its translations. 
Since $C\cap(C+t)\ne\emptyset$ if and only if $t\in C-C$, where the \emph{algebraic difference} of two sets $A, B\subset\R$ is defined by $A-B:=\set{a-b: a\in A, b\in B}$, {the thickness (see Definition \ref{def:thickness} below)  can  be used to study the algebraic difference of Cantor sets
(cf.~\cite{Astels_2000,Kraft-1992,Kraft-2000}). }

In this paper, we consider the algebraic differences of sets of univoque bases for given real numbers. 
To be more precise, for $x\in(0, 1]$, let $\u(x)$ be the set of bases $q\in(1, 2]$ such that $x$ has a unique $q$-expansion. 
Then each element of $\u(x)$ is called a \emph{univoque base} of $x$. 
L\"u et al.~\cite{Lu_Tan_Wu_2014} proved that $\u(x)$ is a Lebesgue null set of full Hausdorff dimension. 

We will prove the following result for the  \emph{algebraic sum} and \emph{product} of $\u(x)$ defined respectively by
\begin{equation*}
\u(x)+\la\u(x):=\set{p+\la q: p, q\in\u(x)}\qtq{and} \u(x)\cdot\u(x)^\la:=\set{pq^\la: p, q\in\u(x)}.
\end{equation*}

\begin{theorem}\label{t11}
{For every $x\in(0, 1]$ and every $\lambda\ne 0$}  both the {sum} $\u(x)+\la\u(x)$ and product $\u(x)\cdot\u(x)^\la$ contain an interval.
\end{theorem}
{We mention that the product $\u(x)\cdot\u(x)^\la$ in Theorem \ref{t11} can be  converted to a sum by taking the logarithm and then repeating the construction (see Section \ref{s3} for more details). Hence, we will focus more on the algebraic sum $\u(x)+\la\u(x)$.}    

\begin{remarks}\label{r12} \mbox{}

\begin{itemize}
\item For $\la=-1$  Theorem \ref{t11} states that the algebraic difference $\u(x)-\u(x)$ {and quotient $\u(x)\cdot\u(x)^{-1}$} contain an interval for each $x\in(0, 1]$.

\item For $x=1$ the set $\u:=\u(1)$ is well-studied. 
For example, it has a smallest element $q_{KL}\approx 1.78723$, called the \emph{Komornik-Loreti constant}  (see \cite{Kom-Lor-1998}), and its closure $\overline{\u}$ is a Cantor set (see \cite{Komornik_Loreti_2007}). 
Furthermore, the local Hausdorff dimension of $\u$ is positive (see \cite{Kong_Li_Lv_Vries2016}), i.e., $\dim_H(\u\cap(q-\de, q+\de))>0$ for any $q\in\u$ and $\de>0$. 
Theorem \ref{t11} for $x=1$ and $\la=-1$ states that  the algebraic difference $\u-\u$ {and quotient $\u\cdot\u^{-1}$} contain an interval. 

\item  {The algebraic sum $\u(x)+\la\u(x)$ containing an interval for all $\la\ne 0$}  can also be expressed by saying that for each $x\in(0, 1]$ and for each oblique straight line $L$ passing through $0$, the projection of the product set $\u(x)\times \u(x)=\set{(p, q): p, q\in\u(x)}$ onto $L$ contains an interval for all $x\in(0, 1]$. 
\end{itemize}
\end{remarks}

{We will also show that the same phenomenon occurs for the set of non-matching parameters, recently studied by the first author and Kalle \cite{Daj-Kal-07}. 
Let us introduce for each $\al\in[1,2]$ the map $S_\al: [-1,1]\ra[-1,1]$ by the formula
\begin{equation*}
S_\al(x)=\left\{\begin{array}{lll}
2x+\al,&\qtq{if} & -1\le x<\frac{1}{2},\\
2x,&\qtq{if}& -\frac{1}{2}\le x\le \frac{1}{2},\\
2x-\al,&\qtq{if}&\frac{1}{2}<x\le 1.
\end{array}\right.
\end{equation*}
The parameter $\al$ is called a \emph{matching parameter} if there exists $m\in\N$ such that $S_\al^m(1)=S_\al^m(1-\al)$, and a  \emph{non-matching parameter} otherwise.

If $\al$ is a matching parameter, then the density $h_\al$ of the invariant measure with respect to $S_\al$ is simply a finite sum of indicator functions. 

It was shown in \cite{Daj-Kal-07} that the set $\n$ of all non-matching parameters is a Lebesgue null set of full Hausdorff dimension. 
We prove the following result:

\begin{theorem}\label{t41}
{For every $\la\ne 0$} both the algebraic sum $\n+\la \n$ and product $\n\cdot \n^\la$ contain an interval. 
\end{theorem}}

The paper is organized as follows. 
In Section \ref{s2} we investigate the topological structure of $\u(x)$ and we construct a Cantor subset of $\u(x)$ in a symbolic way. 
In Section \ref{s3}, we prove Theorem \ref{t11}  by using a theorem of Newhouse on the thickness, and its recent improvements by Astels \cite{Astels_2000} (see Lemmas \ref{l31} and \ref{l35} below). 
{Section \ref{s4} is devoted to the proof of Theorem \ref{t41}}.
In the final section we prove that neither the algebraic sum $\u(1)+\u(1)$, nor the product $\u(1)\cdot\u(1)$ is an interval, and we conjecture that both the algebraic difference $\u(1)-\u(1)$ and quotient ${\u(1)}\cdot{\u(1)}^{-1}$ are intervals. 

\section{Topological structure of $\u(x)$}\label{s2}

Given $x\in(0, 1]$, let $\Phi_x$ be the  coding map defined by 
\begin{equation}\label{21}
\begin{split}
\Phi_x:~ (1,2]&\ra  \set{0, 1}^\N;\qquad q \mapsto  (a_i),
\end{split}
\end{equation}
where $(a_i)$ is the \emph{quasi-greedy} $q$-expansion of $x$, i.e., the lexicographically largest $q$-expansion of $x$ not ending with $0^\f$. 
In this paper, we will use lexicographical order $\prec, \lle, \succ$ and $\lge$ between sequences in $\set{0, 1}^\N$  defined in the natural  way. 
The definitions imply that $\Phi_x$ is strictly increasing with respect to this lexicographical order.
{Therefore,   we may define intervals  in terms of their codings via $\Phi_x$. 
For example, the \emph{symbolic  interval}  $[(a_i), (b_i)]$ with $(a_i), (b_i)\in\set{0,1}^\N$ corresponds to  the  closed interval $[p, q]\subset(1,2]$, where   $p=\Phi_x^{-1}((a_i))$ and $q=\Phi_x^{-1}((b_i))$. 
We emphasize that not every sequence in $[(a_i), (b_i)]$ corresponds to a base in $[p, q]$. 
In other words,   $\Phi_x([p, q])$ is a proper subset of $[(a_i), (b_i)]$. }

Set
\begin{equation*}
\us(x):= \set{\Phi_x(q): q\in\u(x)}.
\end{equation*}
Then $\Phi_x$ is a bijection between $\u(x)$ and $\us(x)$. So, instead of looking at the set $\u(x)$ of univoque bases we focus on the symbolic set $\us(x)$ of univoque sequences. 
In \cite{Lu_Tan_Wu_2014},  L\"{u} et al. {proved} that $\u(x)$  has more weight at the right endpoint $q=2$, i.e., 
{$\lim_{\de\ra 0}\dim_H(\u(x)\cap[2-\de, 2])=1,$   {and} for $q\in(1,2)$ we have  $\lim_{\de\ra 0}\dim_H(\u(x)\cap[q-\de, q+\de])<1$.}
 Accordingly, in the symbolic space the cylinder set
\begin{equation*}
C_n(x)=\set{(a_i)\in\us(x): a_1\cdots a_n=x_1\cdots x_n}
\end{equation*}
has the same topological entropy as the whole set $\us(x)$ for any $n\ge 1$, where $(x_i)=\Phi_x(2)$ is the quasi-greedy dyadic expansion of $x$. 
{Here for a set $X\subseteq\set{0,1}^\N$ its topological entropy $h(X)$ is defined by
\[
h(X):=\liminf_{k\ra\f}\frac{\log|B_n(X)|}{k},
\]
where $|B_n(X)|$ denotes the total number of length $n$ blocks appearing in sequences of $X$.}

Motivated by this observation,  we will construct a symbolic Cantor subset $\us_n(x)$ contained in the cylinder set $C_n(x)$ for all large integers $n$. 
In the next section  we will show that the corresponding  Cantor set $\u_n(x)=\Phi^{-1}_x(\us_n(x))$ has a thickness larger than one for all large integers $n$, and implying that $\u_n(x)+\la\u_n(x)$ contains an interval for each $\la\ne 0$. 
Since $\u_n(x)\subset\u(x)$, this will prove Theorem \ref{t11}.   

The following result was implicitly given by L\"{u} et al.~{\cite[Section 4]{Lu_Tan_Wu_2014}}, and we refer to this article for more details.

\begin{lemma}\label{l21}
Fix  $x\in(0,1]$ arbitrarily and set $(x_i):=\Phi_x(2)$. 
There exist $M\in\N\cup\set{0}$ and a strictly increasing  sequence $(N_j)\subset\set{3,4,\ldots}$ such that the following conditions are satisfied for each $N_j$:

\begin{enumerate}[\upshape (i)]
\item we have 
\begin{equation*}
x_{M+N_j}=1\qtq{and}\us_{N_j}(x)\subseteq\us(x),
\end{equation*}
where  {$\us_{N_j}(x)$} is the set of sequences 
\begin{equation*}
x_1\cdots x_{M+N_j}\ep_1\ep_2\cdots 
\end{equation*}
satisfying  
\begin{equation*}
\ep_1=0,\qtq{and}\ep_{n+1}\cdots \ep_{n+N_j}\notin\set{0^{N_j}, 1^{N_j}}\qtq{for all}n\ge 0;
\end{equation*}
\item we have $(c_i)\lge 0^{M}10^\f$ for all sequences $(c_i)\in\us_{N_j}(x)$.
\item we have $((1^{N_j-1}0)^\f)_q\le 1$ for all bases $q\in\Phi_x^{-1}(\us_{N_j}(x))$.
\end{enumerate}
\end{lemma}
{Before proving the lemma we mention that although the sets $\us_{N_j}(x)$ also depend on $M$, we omit this in the notation for simplicity, because in the rest of the paper $x$ and hence $M$ will be fixed.}

\begin{proof}
Note that $(x_i)=\Phi_x(2)$ is the  dyadic expansion of $x$ not ending with $0^\f$. 
We distingush four cases.

\begin{enumerate}[\upshape (a)]
\item If $(x_i)=x_1\cdots x_m 01^\f$ for some $m\ge 0$, then by \cite{Lu_Tan_Wu_2014} we have
\begin{equation*}
x_1\cdots x_m01^{j+2}\,\ep_1\ep_2\cdots\;\in\us(x)
\end{equation*}
for all $j\ge 1$, where $\ep_1=0$, and for $N_j:=j+2\ge 3$ we have $\ep_{n+1}\cdots \ep_{n+N_j}\notin\set{0^{N_j}, 1^{N_j}}$ for all $n\ge 0$. 
This yields (i) and (ii) by taking $M=m+1$.  
Furthermore, for each $q\in\Phi_x^{-1}(\us_{N_j}(x))$ the inequality 
\begin{equation*}
\sum_{i=1}^{N_j}\frac{1}{q^i}<1
\end{equation*}
holds, and hence (iii) follows:
\begin{equation*}
((1^{N_j-1}0)^\f)_q 
=\left(\sum_{i=1}^{N_j-1}\frac{1}{q^i}\right)\left(\sum_{i=0}^{\infty}\frac{1}{q^{iN_j}}\right)
<\left(1-\frac{1}{q^{N_j}}\right)\left(\sum_{i=0}^{\infty}\frac{1}{q^{iN_j}}\right)
=1.
\end{equation*}

\item If $(x_i)=1^\f$, then  $x=1$.   By a similar argument as in (a) it follows that  
\begin{equation*}
1^{j+2}\ep_1\ep_2\cdots\in\us(x)
\end{equation*}
for any $j\ge 1$, where $\ep_1=0$, and for $N_j:=j+2\ge 3$ we have  $\ep_{n+1}\cdots\ep_{n+N_j}\notin\set{0^{N_j}, 1^{N_j}}$ for all $n\ge 0$. 
This proves (i) and (ii) by taking $M=0$. 
Furthermore, for any $q\in\Phi_x^{-1}(\us_{N_j}(x))$ we have 
\begin{equation*}
\sum_{i=1}^{N_j}\frac{1}{q^i}<x=1;
\end{equation*}
this yields (iii) as above.

\item If $(x_i)=1^{r_1}0^{s_1}1^{r_2}0^{s_2}\cdots 1^{r_k}0^{s_k}\cdots$ with $r_k, s_k\ge 1$ for all $k\ge 1$, then by \cite{Lu_Tan_Wu_2014} we deduce that 
\begin{equation*}
1^{r_1}0^{s_1}\cdots 1^{r_{j+2}}0^{s_{j+2}}01\, \ep_1\ep_2\cdots \in\us(x)
\end{equation*}
for all $j\ge 1$, where $\ep_1=0$ and for ${N_j}:=r_1+s_1+\cdots+r_{j+2}+s_{j+2}-2\ge 4$  we have  
$\ep_{n+1}\cdots \ep_{n+{N_j}}\notin\set{0^{N_j}, 1^{N_j}}$ for all $n\ge 0$. 
Therefore, (i) and (ii) follow  by taking  $M=4$.
Furthermore,  (iii) holds as in the preceding cases because
\begin{equation*}
\sum_{i=1}^{N_j}\frac{1}{q^i}<1
\end{equation*}
for all $q\in\Phi_x^{-1}(\us_{N_j}(x))$.

\item If $(x_i)=0^{r_1}1^{s_1}0^{r_2}1^{s_2}\cdots 0^{r_k}1^{s_k}\cdots$ with $r_k, s_k\ge 1$ for all $k\ge 1$, then by \cite{Lu_Tan_Wu_2014} we have
\begin{equation*}
0^{r_1}1^{s_1}\cdots 0^{r_{j+1}}1^{s_{j+1}} 0^{r_{j+2}}01\ep_1\ep_2\cdots \in\us(x)
\end{equation*}
for all $j\ge 1$, where $\ep_1=0$, and for ${N_j}:=s_1+r_2+s_2+\cdots+r_{j+1}+s_{j+1}+r_{j+2}-1\ge 3$ we have  $\ep_{n+1}\cdots \ep_{n+{N_j}}\notin\set{0^{N_j}, 1^{N_j}}$ for all $n\ge 0$. 
This yields (i) and (ii) by taking $M=r_1+3$.
Finally, (iii) holds again because
\begin{equation*}
\sum_{i=1}^{N_j}\frac{1}{q^i}<1
\end{equation*}
for all $q\in\Phi_x^{-1}(\us_{N_j}(x))$.\qedhere
\end{enumerate}
\end{proof}

\begin{remark}\label{r22}
Lemma \ref{l21} does not hold for $x>1$. 
Indeed, Lemma \ref{l21} (i) states that the set $\us(x)$ contains sequences with arbitrarily long blocks of consecutive zeros, and for this $\us(x)$ must contain bases arbitrarily close to $2$: this follows from the usual lexicographic characterization of unique expansions.
However, for $x>1$ the largest base for which $x$ has an expansion is 
{$
q_x:=1+ {1}/{x}<2.
$}
\end{remark}

By Lemma \ref{l21} the tails of the sequences in $\us_{N_j}(x)$ contain neither $N_j$ consecutive zeros, nor $N_j$ consecutive ones. 
Furthermore,  $\us_{N_j}(x)\subseteq\us(x)$ for all $x\in (0,1]$ and $j\ge 1$. 
Setting
\begin{equation*}
\u_{N_j}(x):=\Phi_x^{-1}(\us_{N_j}(x))=\set{q\in(1,2]: \Phi_x(q)\in\us_{N_j}(x)}
\end{equation*}
we have
\begin{equation}\label{22}
\u_{N_j}(x)\subseteq\u(x)
\end{equation}
for all $x\in (0,1]$ and $j\ge 1$. 
Hence the algebraic sum {$\u(x)+\la\u(x)$ containing an interval} will follow if we prove that the algebraic sum $\u_{N_j}(x)+\la\u_{N_j}(x)$ contains an interval for any fixed $\la\ne 0$, if $j\ge 1$ is 
sufficiently large. 
For this we will apply the results of Newhouse \cite{Newhouse-1979} and Astels \cite{Astels_2000}. 
Notice that $\u_{N_j}(x)$ is a Cantor set for any $x\in(0,1]$ and $j\ge 1$.  
In order to estimate  the thickness of $\u_{N_j}(x)$  we need to describe its geometrical structure. 
For this we need to find an efficient way to construct $\u_{N_j}(x)$ by successively removing  a sequence of open intervals from a closed interval. 

Fix $x\in(0, 1]$ and $j\ge 1$ arbitrarily.
Since the coding map $\Phi_x$ defined in \eqref{21} is strictly increasing,  each $q\in \u_{N_j}(x)$ may be encoded by a unique sequence $\Phi_x(q)=(a_i)\in\us_{N_j}(x)$. 
Conversely,  each sequence $(a_i)\in\us_{N_j}(x)$ can be decoded to a unique base $q\in\u_{N_j}(x)$. 
Let $(x_i)=\Phi_x(2)$ be the dyadic expansion of $x$ not ending with $0^\f$. 
Suppose that the integer $M$ and the sequence $(N_j)$ depending on $x$ are defined as in Lemma  \ref{l21}.
Given $j\ge 1$, let $\Omega_j(x)$ be the set of all finite initial words of length larger than $M+N_j$ occurring in $\us_{N_j}(x)$, i.e., 
\begin{equation*}
\Omega_j(x)=\set{\om_1\cdots \om_n: n>M+N_j\qtq{and}\om_1\cdots\om_n c_1c_2\cdots\in\us_{N_j}(x)\qtq{for some}(c_i)}.
\end{equation*}
Since the tails of  the sequences in $\us_{N_j}(x)$  contain neither  $N_j$ consecutive zeros, nor $N_j$ consecutive ones, the words of $\Omega_j(x)$ are divided into $2N_j-2$ disjoint classes: the words ending with $10^k$ and those ending with $01^k$ for some $k\in\set{1,2,\ldots, N_j-1}$.

{Recall that a symbolic interval $[(a_i), (b_i)]$ corresponds to the closed interval $[p, q]$, if $(a_i)=\Phi_x(p)$ and $(b_i)=\Phi_x(q)$.} For each $\omega\in\Omega_j(x)$ we denote by $\I_\omega$ the smallest {symbolic} interval containing all sequences of $\us_{N_j}(x)$ that begin with $\omega$. The following explicit description of these intervals follows directly from the definition of $\us_{N_j}(x)$.
 
\begin{lemma}\label{l23}
Let $\omega\in\Omega_j(x)$.
\begin{enumerate}[\upshape (i)]
\item If $\omega$ ends with $10^k$ for some $k\in\set{1,\ldots, N_j-1}$, then 
\begin{equation*}
\I_{\omega}=\big[\omega 0^{N_j-1-k}(10^{N_j-1})^\f,\; \omega(1^{N_j-1}0)^\f\big].
\end{equation*}

\item If $\omega$ ends with $01^k$ for some $k\in\set{1,\ldots, N_j-1}$, then 
\begin{equation*}
\I_\omega=\big[\omega(0^{N_j-1}1)^\f,\; \omega1^{N_j-1-k}(01^{N_j-1})^\f\big].
\end{equation*}
\end{enumerate}
\end{lemma}

By Lemma \ref{l21} {(i)} all sequences in $\us_{N_j}(x)$ begin with ${x_1\cdots x_{M+N_j}0=}x_1\cdots x_{M+N_j-1}\,10$.
Applying Lemma \ref{l23} (i) it follows that the smallest {symbolic} interval which contains $\us_{N_j}(x)$ is 
\begin{equation*}
\I_{x_1\cdots x_{M+N_j}0}=\big[x_1\cdots x_{M+N_j}(0^{N_j-1}1)^\f,\; x_1\cdots x_{M+N_j}(01^{N_j-1})^\f\big].
\end{equation*}
An immediate consequence of Lemma \ref{l23} is the following:

\begin{lemma}
\label{l24}
Let $\omega\in\Omega_j(x)$.
\begin{enumerate}[\upshape (i)]
\item If $\omega$ ends with $10^{N_j-1}$, then 
\begin{equation*}
\omega 0\notin\Omega_j(x)\qtq{and} \I_{\omega 1}=\I_\omega.
\end{equation*}

\item If $\omega$ ends with $01^{N_j-1}$, then 
\begin{equation*}
\omega 1\notin\Omega_j(x)\qtq{and}\I_{\omega 0}=\I_\omega.
\end{equation*}

\item In the remaining cases, $\I_\omega$ is the disjoint union of the non-empty intervals 
\begin{equation*}
\I_{\omega 0}, \quad \I_{\omega 1}\qtq{and}\G_\omega:=\I_\omega\setminus(\I_{\omega 0}\cup \I_{\omega 1}).
\end{equation*}
\end{enumerate}
\end{lemma}

Now we may describe the geometrical structure of $\u_{N_j}(x)$. 
Given a {symbolic}  interval $\I=[(a_i), (b_i)]$ with $(a_i), (b_i)\in \us_{N_j}(x)$, we denote by $I=[p, q]$ the corresponding interval in $\R$, where $p=\Phi_x^{-1}((a_i))$ and $q=\Phi_x^{-1}((b_i))$. 
Then the symbolic  intervals $\I_\om, \G_\om$ are transferred  to the real intervals $I_\om, G_\om$, respectively.    
Set
\begin{equation*}
\Omega_j^*(x):=\set{\omega\in\Omega_j(x): G_\om\ne\emptyset}.
\end{equation*}

\begin{lemma}\label{l25}
The non-empty open intervals $G_\omega, \omega\in\Omega_j^*(x)$ are pairwise disjoint, and 
\begin{equation*}
\u_{N_j}(x)=I_{x_1\cdots x_{M+N_j} 0}\setminus\bigcup_{\omega\in\Omega_j^*(x)} G_\omega.
\end{equation*}
\end{lemma}

\begin{proof}
The map $\Phi_x: \u_{N_j}(x)\ra \us_{N_j}(x)$ is strictly increasing, hence bijective. 
Lemmas \ref{l21}--\ref{l24} imply that
\begin{equation*}
\u_{N_j}(x)\subseteq I_{x_1\cdots x_{M+N_j} 0}\setminus\bigcup_{\omega\in\Omega_j^*(x)} G_\omega.
\end{equation*}

For the converse inclusion,  first we remove from the closed interval $I_{x_1\cdots x_{M+N_j}0}$ the non-empty open  interval $G_{x_1\cdots x_{M+N_j} 0}$ to obtain the union of two non-degenerate disjoint closed intervals $I_{x_1\cdots x_{M+N_j}00}$ and $I_{x_1\cdots x_{M+N_j}01}$. {We emphasize that the non-empty of $G_{x_1\ldots x_{M+N_j}0}$ follows by Lemma \ref{l24}, since $N_j\ge 3$ and the word $x_1\ldots x_{M+N_j}0$ ends with $10$ by Lemma \ref{l21}.}
Then we proceed by induction. 
Assume that after a finite number of steps we get a disjoint union of non-degenerate closed intervals $I_{\omega}$, where $\omega$ runs over all length $n (>M+N_j)$ words of $\Omega_j(x)$. 
We will construct all level $n+1$ sub-intervals in the following way.
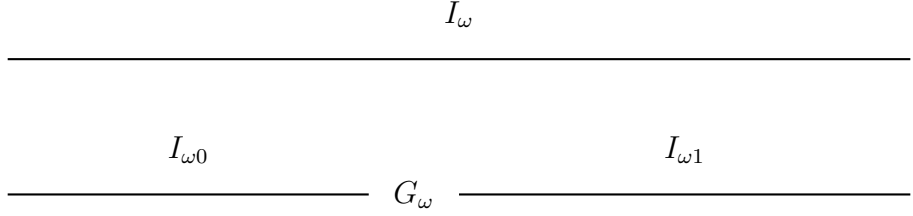
\begin{figure}[h!]
\begin{center}
\begin{tikzpicture}[
    scale=12,
    axis/.style={very thick, ->},
    important line/.style={thick},
    dashed line/.style={dashed, thin},
    pile/.style={thick, ->, >=stealth', shorten <=2pt, shorten
    >=2pt},
    every node/.style={color=black}
    ]
    
    \draw[important line]  (0,0.85)--( 1.0, 0.85);
    \node[] at(0.5, 0.9){$I_\om$}; 
    \draw[important line] (0, 0.7)--(0.4, 0.7);
     \node[] at(0.2, 0.75){$I_{\om 0}$}; 
    \draw[important line] (0.5, 0.7)--(1, 0.7);
          \node[] at(0.75, 0.75){$I_{\om 1}$};\node[] at(0.45, 0.7){$G_\om$};

\end{tikzpicture} 
\end{center}
\caption{The geometrical structure of the basic intervals $I_\om, I_{\om 0}, I_{\om 1}$ and the gap interval $G_\om$.}\label{fig:1}
\end{figure}
If $\omega\in\Omega_j^*(x)$, then we remove the open interval $G_\omega$, and  replace $I_\omega$ by the two disjoint  closed subintervals $I_{\omega 0}$ and $I_{\omega 1}$ (see Figure \ref{fig:1}). If $\omega\notin\Omega_j^*(x)$, then either $\omega 0\in\Omega_j(x)$ or $\omega 1\in\Omega_j(x)$. In this case  we keep the interval $I_\omega$ with  either   $I_{\omega}=I_{\omega 0}$ or $I_{\omega}=I_{\omega 1}$.  

Repeating this procedure  indefinitely  we construct the set $\u_{N_j}(x)$, and we obtain the converse inclusion
\begin{equation*}
I_{x_1\cdots x_{M+N_j} 0}\setminus\bigcup_{\omega\in\Omega_j^*(x)} G_\omega\subseteq \u_{N_j}(x).
\end{equation*}
Furthermore, we obtain that the gap intervals $G_\om$ with $ \om\in\Omega_j^*(x)$ are pairwise disjoint.  
\end{proof}

\section{Proof of Theorem \ref{t11}}\label{s3}

By Lemma \ref{l25} the Cantor set $\u_{N_j}(x)$ can be obtained by successively  removing from the closed interval $I_{x_1\cdots x_{M+N_j}0}$ a sequence of open intervals. By using the notation from Lemma \ref{l25} we define the thickness of $\u_{N_j}(x)$.
{\begin{definition}\label{def:thickness}
The \emph{thickness} of $\u_{N_j}(x)$ is defined by
\begin{equation*}
\tau(\u_{N_j}(x)):=\inf_{\omega\in\Omega_j^*(x)}\set{\frac{| I_{\omega 0}|}{| G_\omega|}, \frac{| I_{\omega 1}|}{|G_{\omega}|}},
\end{equation*}
where $|I|:=q-p$ denotes the length of a interval $I=[p, q]$.
\end{definition}
We point out that the thickness given in Definition \ref{def:thickness} coincides with that defined by Astels \cite{Astels_2000}, and it is essentially the same as that defined by Newhouse \cite{Newhouse-1979}.}
Notice that the thickness is stable under non-trivial scaling, i.e., $\tau(\la\u_{N_j}(x))=\tau(\u_{N_j}(x))$ for all $\la\ne 0$. 
The following result follows from \cite[Theorem 2.4]{Astels_2000}.

\begin{lemma}\label{l31}
If $\tau(\u_{N_j}(x))\ge 1$, then $\u_{N_j}(x)+\la \u_{N_j}(x)$ contains an interval for all $\la\ne 0$.
\end{lemma}

In view of the relation \eqref{22} and  Lemma \ref{l31}, {the algebraic sum $\u(x)+\la\u(x)$ containing an interval} will be proved if we find an index $j\ge 1$ such that $\tau(\u_{N_j}(x))\ge 1$.
For this we will compare the length of each  non-degenerate  interval $G_\omega$ with the lengths of its neighbors $I_{\omega 0}$ and $I_{\omega 1}$. 
We need three further lemmas; for the first one see also  \cite{Kong_2016}.

Henceforth we denote by $\varphi:=\frac{1+\sqrt{5}}{2}$ the Golden Ratio.
\begin{lemma}\label{l32}
We have  $\u(x)\subseteq(\varphi ,2]$ for all $x\in(0, 1]$.
\end{lemma}

\begin{proof}
For $q\in(1, \varphi ]$ only the endpoints of $[0, 1/(q-1)]$ have unique expansions, and they are outside $(0, 1]$.
\end{proof}

Next we establish some  elementary inequalities.

\begin{lemma}\label{l33}
If the integers $m$ and $n$ are sufficiently large, then
\begin{equation*}
\left(1+\frac{1}{\varphi ^{m}}\right)^{2 m}<\frac{(110^\f)_2}{((10^{n-1})^\f)_{\varphi }}\qtq{and} \left(1+\frac{1}{\varphi ^{m}}\right)^{2 m}<\frac{((1^{n-1}0)^\f)_2}{((10^{n-3}10)^\f)_{\varphi }}.
\end{equation*}
\end{lemma}

\begin{proof}
The lemma follows from the following relations:
\begin{align*}
&\lim_{m\ra\f}\left(1+\frac{1}{\varphi ^{m}}\right)^{2m}=1,\\
&\lim_{n\ra\f}((10^{n-1})^\f)_{\varphi }=\frac{1}{\varphi }<\frac{3}{4}=(110^\f)_2 
\intertext{and} 
&\lim_{n\ra\f}((10^{n-3}10)^\f)_{\varphi } =\frac{1}{\varphi }<1=\lim_{n\ra\f}((1^{n-1}0)^\f)_2. \qedhere
\end{align*}
\end{proof}
 
\begin{lemma}\label{l34}
Let $j\ge 1$ be sufficiently large. Then 
\begin{equation*}
| G_\omega|\le | I_{\om 0}|\qtq{and} | G_\omega|\le | I_{\omega 1}|
\end{equation*}
for all $\omega\in\Omega_j^*(x)$. 
\end{lemma}

\begin{proof}
Fix    $\omega\in\Omega_j^*(x)$ of length $n(>M+N_j)$. Writing
\begin{equation*}
I_{\omega 0}=[q_1, q_2] \qtq{and} I_{\omega 1}=[q_3, q_4]
\end{equation*}
we have to prove the inequalities
\begin{equation*}
q_3-q_2\le q_2-q_1\qtq{and} q_3-q_2\le q_4-q_3
\end{equation*}
for some large integer $j$.
By Lemma \ref{l23} it follows that 
\begin{equation}\label{31}
\begin{split}
&\omega (0^{N_j-1}1)^\f\lle \Phi_x(q_1)\lle \omega 0(10^{N_j-1})^\f,\quad \Phi_x(q_2)=\omega 0(1^{N_j-1} 0)^\f;\\
&\omega 1(01^{N_j-1})^\f\lle \Phi_x(q_4)\lle \omega (1^{N_j-1}0)^\f,\quad \Phi_x(q_3)=\omega 1(0^{N_j-1}1)^\f.
\end{split}
\end{equation}
We emphasize by Lemma \ref{l25} that $q_i\in\u_{N_j}(x)$ for all $1\le i\le 4$.
\medskip 

\emph{Bounds on  $q_2-q_1$.}
First we give an \emph{upper} {bound of $q_2-q_1$}. 
It follows from \eqref{31} that 
\begin{equation*}
(\omega(01^{N_j-1})^\f)_{q_2}=x\ge(\omega(0^{N_j-1}1)^\f)_{q_1}{,}
\end{equation*}
whence
\begin{equation*}
(0^n(01^{N_j-1})^\f)_{q_2}-(0^n(0^{N_j-1}1)^\f)_{q_1}\ge (\om 0^\f)_{q_1}-(\om 0^\f)_{q_2}.
\end{equation*}
Since $\omega=\om_1\cdots \om_n$ contains a non-zero digit $\om_\ell=1$ for some  $1\le\ell\le M+1$ by Lemma \ref{l21} (ii), the right hand side may be minorized as follows:
\begin{equation*}
(\om 0^\f)_{q_1}-(\om 0^\f)_{q_2}
\ge\frac{1}{q_1^{\ell}}-\frac{1}{q_2^{\ell}}
\ge\frac{1}{q_1q_2^{\ell-1}}-\frac{1}{q_2^\ell}=\frac{q_2-q_1}{q_1q_2^\ell}\ge \frac{q_2-q_1}{q_2^{M+2}}.
\end{equation*}
Combining the two estimates and using  Lemma \ref{l21} (iii) we conclude that 
\begin{equation}\label{32}
\begin{split}
q_2-q_1&\le q_2^{M+2}\left((0^n(01^{N_j-1})^\f)_{q_2}-(0^n(0^{N_j-1}1)^\f)_{q_1}\right)\\
&\le q_2^{M+2} (0^n(01^{N_j-1})^\f)_{q_2}\;\le \frac{q_2^{M+2}}{q_2^{n+1}}= \frac{1}{q_2^{n-M-1}}.
\end{split}
\end{equation}

Now we focus on the lower bound of $q_2-q_1$.
We infer from \eqref{31} that
\begin{equation*}
(\omega 0(1^{N_j-1}0)^\f)_{q_2}=x\le (\omega 0(1 0^{N_j-1})^\f)_{q_1},
\end{equation*}
and this implies the estimate 
\begin{align*}
(0^{n+1}(1^{N_j-1}0)^\f)_{q_2}-(0^{n+1}(10^{N_j-1})^\f)_{q_1}&\le (\omega 0^\f)_{q_1}-(\omega 0^\f)_{q_2}\\
&\le \sum_{i=1}^\f\left(\frac{1}{q_1^i}-\frac{1}{q_2^i}\right)=\frac{q_2-q_1}{(q_1-1)(q_2-1)}.
\end{align*}
Choosing by Lemma \ref{l33} a large integer  $j_0\ge 1$ such that
\begin{equation}\label{33}
 N_j\ge 4\qtq{and}
\left(1+\frac{1}{\varphi ^{n-M}}\right)^{n+1}<\frac{(110^\f)_2}{((10^{N_j-1})^\f)_{\varphi }}
\end{equation}
for all $j\ge j_0$ and $n>M+N_j$, we deduce from the above estimate for all $j\ge j_0$ that
\begin{equation}\label{34}
\begin{split}
q_2-q_1&\ge (\varphi -1)^2\left((0^{n+1}(1^{N_j-1}0)^\f)_{q_2}-(0^{n+1}(10^{N_j-1})^\f)_{q_1}\right)\\
&\ge (\varphi -1)^2\left((0^{n+1}(1^{N_j-1}0)^\f)_{q_2}-(0^{n+1}110^\f)_{q_2}\right) \\
&\ge \frac{(\varphi -1)^2}{q_2^{n+4}}.
\end{split}
\end{equation}
Here the first inequality holds because $q_2>q_1\ge \varphi $ by Lemma \ref{l32} and the last inequality holds because $N_j\ge 4$.
The crucial second inequality follows by \eqref{32}, \eqref{33}  and the inequality $q_2>q_1\ge \varphi $:
\begin{align*}
(0^{n+1}(10^{N_j-1})^\f)_{q_1}&=\left(\frac{q_2}{q_1}\right)^{n+1}\frac{((10^{N_j-1})^\f)_{q_1}}{q_2^{n+1}}\\
&\le\left(1+\frac{q_2-q_1}{q_1}\right)^{n+1}\frac{((10^{N_j-1})^\f)_{\varphi }}{q_2^{n+1}}\\
&\le\left(1+\frac{1}{q_1q_2^{n-M-1}}\right)^{n+1}\frac{((10^{N_j-1})^\f)_{\varphi }}{q_2^{n+1}}\\
&\le \left(1+\frac{1}{\varphi ^{n-M}}\right)^{n+1}\frac{((10^{N_j-1})^\f)_{\varphi }}{q_2^{n+1}}\\
&<\frac{(110^\f)_{2}}{q_2^{n+1}}\le\; (0^{n+1}110^\f)_{q_2}.
\end{align*}
\medskip 

\emph{Bounds on $q_4-q_3$.}
We adapt the above arguments for $q_2-q_1$.
First we give an upper bound of $q_4-q_3$. 
We infer from \eqref{31}  that 
\begin{equation*}
(\om 1(0^{N_j-1}1)^\f)_{q_3}=x\le(\om (1^{N_j-1}0)^\f)_{q_4}.
\end{equation*}
Since there exists $1\le \ell\le M+1$ such that $\om_\ell=1$ by Lemma \ref{l21} (ii), it follows that 
\begin{align*}
(0^{n+1}(1^{N_j-2}01)^\f)_{q_4}-(0^{n+1}(0^{N_j-1}1)^\f)_{q_3}&\ge(\om 10^\f)_{q_3}-(\om 10^\f)_{q_4}\\
&\ge\frac{1}{q_3^\ell}-\frac{1}{q_4^\ell}\ge\frac{q_4-q_3}{q_4^{M+2}}.
\end{align*}
This implies that 
\begin{equation}\label{35}
\begin{split}
q_4-q_3&\le q_4^{M+2}\left((0^{n+1}(1^{N_j-2}01)^\f)_{q_4}-(0^{n+1}(0^{N_j-1}1)^\f)_{q_3}\right)\\
&\le q_4^{M+2}(0^{n+1}(1^{N_j-2}01)^\f)_{q_4} \;
 \le \frac{q_4^{M+2}}{q_4^{n+1}}=\frac{1}{q_4^{n-M-1}},
\end{split}
\end{equation}
where the third inequality follows by Lemma \ref{l21} (iii) because $q_4\in\u_{N_j}(x)$.

Now we seek a lower bound of $q_4-q_3$. 
By Lemma \ref{l33} there exists $j_1\ge j_0$ (we use $j_0$ chosen in the first part of the proof) such that 
\begin{equation}\label{36}
\left(1+\frac{1}{\varphi ^{n-M}}\right)^{n+2}<\frac{((1^{N_j-1}0)^\f)_{2}}{((10^{N_j-3}10)^\f)_{\varphi }}
\end{equation}
for all $j\ge j_1$ and  $n>M+N_j$.
By \eqref{31} we have
\begin{equation*}
(\omega 1(0^{N_j-1}1)^\f)_{q_3}=x\ge (\omega 1(01^{N_j-1})^\f)_{q_4},
\end{equation*}
whence
\begin{align*}
(0^{n+1}(01^{N_j-1})^\f)_{q_4}-(0^{n+1}(0^{N_j-1}1)^\f)_{q_3}&\le (\omega 10^\f)_{q_3}-(\omega 10^\f)_{q_4}\\
&\le \sum_{i=1}^\f\left(\frac{1}{q_3^i}-\frac{1}{q_4^i}\right)=\frac{q_4-q_3}{(q_4-1)(q_3-1)}.
\end{align*}
Since $q_4> q_3\ge \varphi $ by Lemma \ref{l32}, hence we deduce the following  estimate of $q_4-q_3$ for all $j\ge j_1$:
\begin{equation}
\label{37}
\begin{split}
q_4-q_3&\ge (\varphi -1)^2\left((0^{n+1}(01^{N_j-1})^\f)_{q_4}-(0^{n+1}(0^{N_j-1}1)^\f)_{q_3}\right)\\
&\ge (\varphi -1)^2\left((0^{n+1}(010^{N_j-3}1)^\f)_{q_3}-(0^{n+1}(0^{N_j-1}1)^\f)_{q_3}\right)\\
&\ge \frac{(\varphi -1)^2}{q_3^{n+3}}.
\end{split}
\end{equation}
Here the crucial second inequality follows from \eqref{35} and \eqref{36}:
\begin{align*}
(0^{n+1}(010^{N_j-3}1)^\f)_{q_3}&=\left(\frac{q_4}{q_3}\right)^{n+2}\frac{((10^{N_j-3}10)^\f)_{q_3}}{q_4^{n+2}}\\
&\le\left(1+\frac{q_4-q_3}{q_3}\right)^{n+2}\frac{((10^{N_j-3}10)^\f)_{\varphi }}{q_4^{n+2}}\\
&\le\left(1+\frac{1}{q_3q_4^{n-M-1}}\right)^{n+2}\frac{((10^{N_j-3}10)^\f)_{\varphi }}{q_4^{n+2}}\\
&\le\left(1+\frac{1}{\varphi ^{n-M}}\right)^{n+2}\frac{((10^{N_j-3}10)^\f)_{\varphi }}{q_4^{n+2}}\\
&<\frac{((1^{N_j-1}0)^\f)_2}{q_4^{n+2}}\le\;(0^{n+1}(01^{N_j-1})^\f)_{q_4}.
\end{align*}
\medskip 

\emph{Bounds on $q_3-q_2$.}
Note that
\begin{equation*}
(\omega 0(1^{N_j-1}0)^\f)_{q_2}=x=(\omega 1(0^{N_j-1}1)^\f)_{q_3}
\end{equation*}
by \eqref{31}.
Since there exists $1\le\ell\le M+1$ such that $\om_\ell=1$ by Lemma \ref{l21} (ii), it follows that
\begin{equation*}
(0^n1(0^{N_j-1}1)^\f)_{q_3}-(0^n 0(1^{N_j-1}0)^\f)_{q_2}
=(\omega 0^\f)_{q_2}-(\omega 0^\f)_{q_3}
\ge \frac{1}{q_2^{\ell}}-\frac{1}{q_3^{\ell}} \ge \frac{q_3-q_2}{ q_3^{M+2}}.
\end{equation*}
Using the inequalities $q_2<q_3\le 2$ hence we infer that
\begin{equation}
\label{38}
\begin{split}
q_3-q_2&\le 2^{M+2}\left((0^n 1(0^{N_j-1}1)^\f)_{q_3}-(0^n 0(1^{N_j-1}0)^\f)_{q_2}\right)\\
&\le 2^{M+2}\left((0^n1(0^{N_j-1}1)^\f)_{q_3}-(0^n0(1^{N_j-1}0)^\f)_{q_3}\right)\\
&\le 2^{M+2}\left((0^n0 1^{N_j-1}4 0^\f)_{q_3}-(0^n 0 1^{N_j-1}0^\f)_{q_3}\right)\\
&=\frac{ 2^{M+4}}{q_3^{n+N_j+1}}{\color{blue}.}
\end{split}
\end{equation}
{Here the crucial  third inequality follows by
\[
(0^n1(0^{N_j-1}1)^\f)_{q_3}<(0^{n+1}(1^{N_j-1}2)^\f)_{q_3}
\]
and the estimate
\begin{align*}
((1^{N_j-1}2)^\f)_{q_3}&=\frac{(1^{N_j-1}20^\f)_{q_3}}{1-q_3^{-N_j}}\le \frac{1+q_3^{-N_j}}{1-q_3^{-N_j}}\le \frac{1+\varphi^{-N_j}}{1-\varphi^{-N_j}}\;\le 2,
\end{align*}
using that
$(1^{N_j}0^\f)_{q_3}\le 1$, $q_3\ge \varphi$ and $N_j\ge 3$.
}

Since $1<q_2<q_3$, we may  choose $j_2\ge j_1$ such that
\begin{equation*}
2^{M+4}\le (\varphi -1)^2q_2^{N_{j_2}-3}\le (\varphi -1)^2q_3^{N_{j_2}-2}.
\end{equation*}
(The second inequality automatically follows from the first one.)
Then, using also the relations \eqref{34} and \eqref{38}, the following estimate holds for all $j\ge j_2$:
\begin{equation*}
q_3-q_2\le \frac{2^{M+4}}{q_3^{n+N_j+1}}<\frac{2^{M+4}}{q_2^{n+N_j+1}}\le \frac{(\varphi -1)^2}{q_2^{n+4}}\le q_2-q_1.
\end{equation*}
Similarly, using \eqref{37} and \eqref{38} we obtain that
\begin{equation*}
q_3-q_2\le \frac{2^{M+4}}{q_3^{n+N_j+1}}<\frac{(\varphi -1)^2}{q_3^{n+3}}\le q_4-q_3
\end{equation*}
for all $j\ge j_2$.  
Since the word $\om$ was taken arbitrarily from $\Omega_j^*(x)$, this completes the proof.
\end{proof}

{Now we consider  the algebraic product  part of Theorem \ref{t11}. By Lemma \ref{l32} we have $\u(x)\subset(\varphi, 2]$ for each $x\in(0,1]$. Then 
\[
\u(x)\cdot\u(x)^\la=\set{pq^\la: p,q\in\u(x)}=\set{e^{\ln p+\la\ln q}: p, q\in\u(x) }.
\]
So,  the algebraic product $\u(x)\cdot\u(x)^\la$ containing an interval is equivalent to that the algebraic sum $\ln\u(x)+\la\ln\u(x)$ contains an interval, where $\ln\u(x):=\set{\ln q: q\in\u(x)}$. Observe by Lemma \ref{l25} that for any $x\in(0, 1]$ and  any $j\ge 1$ the set $\u_{N_j}(x)$ is a Cantor subset of $\u(x)$. This implies that $\ln\u_{N_j}(x)$ is also a Cantor subset of $\ln\u(x)$. Combining this with Lemma \ref{l31} on the thickness we obtain the following
\begin{lemma}
\label{l35}
For any given $x\in(0, 1]$, if $\tau\big(\ln\u_{N_j}(x)\big)\ge 1$ for some $j\ge 1$, then  $\u_{N_j}(x)\cdot \u_{N_j}(x)^\la$ contains an interval for each non-zero real number $\la$.  
\end{lemma}}
\begin{proof}[Proof of Theorem \ref{t11}]
{Fix $x\in(0, 1]$ and $\la\ne 0$ arbitrarily. By Lemmas \ref{l31} and \ref{l34} it follows that the algebraic sum $\u(x)+\la\u(x)$ contains an interval.  As for the algebraic product $\u(x)\cdot \u(x)^\la$}
it suffices to show that $\tau\big(\ln\u_{N_j}(x)\big)\ge 1$ if $j$ is sufficiently large.
Indeed, then the theorem will follow from Lemma \ref{l35} because of the inclusion \eqref{22}.

Fix $\om\in\Omega^*_j(x)$ arbitrarily, of length $n(>M+N_j)$, and consider the intervals 
\begin{equation*}
I_{\om 0}=[q_1, q_2],\quad   I_{\om 1}=[q_3, q_4]\qtq{and}G_\om=(q_2, q_3)
\end{equation*}
as in the proof of Lemma \ref{l34}. 
Then the corresponding basic intervals of level $n+1$ of $\ln(\u_{N_j}(x))$ are 
\begin{equation*}
\ln( I_{\om 0}):=[\ln q_1, \ln q_2],\quad \ln( I_{\om 1}):=[\ln q_3, \ln q_4]\qtq{and}\ln(G_\om):=(\ln q_2, \ln q_3).
\end{equation*}
We have to prove that if $j$ is sufficiently large, then
\begin{equation*}
\ln q_3-\ln q_2\le \ln q_2-\ln q_1\qtq{and}\ln q_3-\ln q_2\le \ln q_4-\ln q_3,
\end{equation*}
or equivalently 
\begin{equation}\label{39}
\frac{q_3}{q_2}\le \frac{q_2}{q_1}\qtq{and}\frac{q_3}{q_2}\le \frac{q_4}{q_3}.
\end{equation}

We use the estimates obtained in the proof of {Lemma \ref{l34}}.
If $j\ge j_2$, then we infer from \eqref{34} and \eqref{38} the relations
\begin{align*}
\frac{q_2}{q_1}&\ge 1+\frac{(\varphi -1)^2}{q_1q_2^{n+4}}\ge1+\frac{(\varphi -1)^2}{q_2^{n+5}},\\
\frac{q_3}{q_2}&\le 1+\frac{2^{M+4}}{q_2q_3^{n+N_j+1}}\le 1+\frac{2^{M+4}}{q_2^{n+N_j+2}}.
\end{align*}
Hence there exists $j_3\ge j_2$ such that 
\begin{equation*}
\frac{q_3}{q_2}\le 1+\frac{2^{M+4}}{q_2^{n+N_j+2}}<1+\frac{(\varphi -1)^2}{q_2^{n+5}}\le \frac{q_2}{q_1}
\end{equation*}
for all $j\ge j_3$, establishing the first inequality in \eqref{39}.

Similarly, we deduce from \eqref{37} and \eqref{38}  that 
\begin{align*}
\frac{q_4}{q_3}&\ge1+\frac{(\varphi -1)^2}{q_3^{n+4}}\qtq{and}
\frac{q_3}{q_2} \le 1+\frac{2^{M+4}}{q_2 q_3^{n+N_j+1}}
\end{align*}
for all $j\ge j_2$. 
Hence, there exists $j_4\ge j_3$ such that 
\begin{equation*}
\frac{q_3}{q_2}\le 1+\frac{2^{M+4}}{q_2q_3^{n+N_j+1}}<1+\frac{(\varphi -1)^2}{q_3^{n+4}}\le \frac{q_4}{q_3}
\end{equation*}
for all $j\ge j_4$.
This proves the second inequality in \eqref{39}.
\end{proof}

\section{Proof of Theorem \ref{t41}}\label{s4}

{In this section we apply the symbolic Cantor sets constructed in Section \ref{s2} to the set $\n$ of non-matching parameters, and we  prove Theorem \ref{t41}.}
In order to describe the non-matching set $\n$ we recall the doubling map $D$ on the unit circle $[0, 1)$ {defined by}
\begin{equation*}
D: [0, 1)\ra[0, 1);\qquad x\mapsto 2x\pmod 1.
\end{equation*}
The following characterization of $\n$ was implicitly given by \cite{Daj-Kal-07}.

\begin{lemma}
\label{l42}
The following statements are equivalent:
\begin{enumerate}[\upshape (i)]
\item $\al\in\n$.
\item  For all $n\ge 0$ we have 
\[D^n\left(\frac{1}{\al}\right)\notin\left(\frac{1}{2\al}, 1-\frac{1}{2\al}\right).\]
\item {$1/\alpha\in[1/2, 1]$ has a unique dyadic expansion $(a_i)\in\set{0,1}^\N$ satisfying  
\begin{equation}\label{41}
\left\{\begin{array}{lll}
a_{n+1}a_{n+2}\cdots\lle a_1a_2\cdots&\quad\text{if}& a_n=0,\\
a_{n+1}a_{n+2}\cdots\lge (1-a_1)(1-a_2)\cdots&\quad\text{if}& a_n=1
\end{array}\right.
\end{equation}
for all $n\ge 1$.
}
\end{enumerate}
\end{lemma}
{\begin{proof}
The equivalence of (i) and (ii) follows from \cite{Daj-Kal-07}. As for (iii) $\Rightarrow$ (ii),  let $(a_i)$ be the unique dyadic expansion of $1/\alpha$. Then $(1-a_i)$ is   the unique dyadic expansion of $1-1/\alpha$. Hence, (ii) follows from (\ref{41}).

 To prove (ii) $\Rightarrow$ (iii), we first observe that the greedy dyadic expansion $(a_i)$ of $1/\alpha$ {cannot} end with $10^\f$, {for} otherwise there must exist $n\ge 0$ such that 
\[D^n\left(\frac{1}{\alpha}\right)=\frac{1}{2}\in\left(\frac{1}{2\alpha}, 1-\frac{1}{2\alpha}\right).\]
Hence, $1/\alpha$ has a unique dyadic expansion $(a_i)$. Furthermore, (4.1) follows from the   following observation: for each $n\ge 1$,
\[
D^{n-1}\left(\frac{1}{\alpha}\right)\le \frac{1}{2\alpha}\quad\Longleftrightarrow \quad a_{n}=0~\textrm{ and }~a_{n+1}a_{n+2}\ldots \lle a_1a_2\ldots
\]
and 
\[
D^{n-1}\left(\frac{1}{\alpha}\right)\ge 1-\frac{1}{2\alpha}\quad\Longleftrightarrow \quad a_{n}=1~\textrm{ and }~a_{n+1}a_{n+2}\ldots \lge (1-a_1)(1-a_2)\ldots. \qedhere
\]
\end{proof}
 
Let $\ns$ be the set of all sequences $(a_i)\in\set{0, 1}^\N$ such that it is the unique dyadic expansion of $((a_i))_2\in[1/2, 1]$ and it satisfies  the inequalities in \eqref{41}.} Then by Lemma \ref{l42} it follows that the projection map
\begin{equation*}
\Psi: \ns\ra\n;\qquad (a_i)\mapsto\frac{1}{((a_i))_2}
\end{equation*}
is well-defined. 
Indeed, $\Psi$ is bijective and  strictly decreasing. 
Motivated by the symbolic Cantor sets constructed in Section \ref{s2}, we will construct the symbolic Cantor subsets $\ns_m$ contained in $\ns$, such that the thickness of $\Psi(\ns_m)$ is larger than $1$.   

Given {an} integer $m\ge {3}$, let $\ns_m$ be the set of sequences $(a_i)\in\set{0, 1}^\N$ satisfying
\begin{equation*}
a_1\cdots a_m=1^m\qtq{and} a_{n+1}\cdots a_{n+m}\notin\set{0^m, 1^m}
\end{equation*}
for all $n\ge m$. {Then each sequence $(a_i)\in\ns_m$ satisfies (\ref{41}) and   ends with neither $01^\f$ nor $10^\f$. Hence, by Lemma \ref{l42} it follows that}
\begin{equation*}
\ns_m\subseteq\ns\qtq{for all} m\ge {3}.
\end{equation*}
By an analogous argument as in Lemmas \ref{l23}--\ref{l25}, the set $\ns_m$ is indeed a symbolic Cantor set and has a similar structure as $\us_{N_j}(x)$ as described in Section \ref{s2}. 
Write  $\n_m:=\Psi(\ns_m)$. 
By Lemma \ref{l42} it follows that $\n_m\subset\n$ for all $m\ge {3}$. 
Therefore it suffices to prove the thickness $\tau(\n_m)\ge 1$ for some large integer $m$. 

In contrast with the definitions of the set $\Omega_j(x)$ of finite words and the symbolic  intervals $\I_\om$ in Section \ref{s2}, we introduce the following notation.
For $m\ge {3}$, let $\Omega(\ns_m)$ be the set of all finite initial words of length larger than $m$ occurring in $\ns_m$. 
Given a  word  $\om\in\Omega(\ns_m)$, let $\J_\om$ be the smallest symbolic  interval containing all sequences of $\ns_m$ that begin with $\om$. 
{Similarly} to Lemma \ref{l23}, one can verify that  the  interval $\J_\om$ has the form $\J_\om=[(a_i), (b_i)]$ with  $(a_i), (b_i)\in\ns_m$. 
Notice that the map $\Psi$ is strictly decreasing on $\ns_m$.  
Then we denote by $J_\om=[p, q]$ the corresponding  interval in $\R$, where  $p=\Psi((b_i))$ and $q=\Psi((a_i))$.   
\begin{proof}[Proof of Theorem \ref{t41}]
Fix a  word $\om\in\Omega(\ns_m)$ of length $n(>m)$ such that the  open interval  $ O_\om:= J_\om\setminus( J_{\om 0}\cup  J_{\om 1})\ne\emptyset$.   Write
\begin{equation*}
J_\om= J_{\om 1}\cup  O_\om\cup  J_{\om 0}=:[p_1, p_2]\cup(p_2, p_3)\cup[p_3, p_4].
\end{equation*}
Notice that the   map $\Psi$ is   strictly decreasing.  
By Lemma \ref{l23} it follows that  
\begin{equation}\label{42}
\begin{split}
&\Psi(\om(1^{m-1}0)^\f)\le p_1\le \Psi(\om 1(01^{m-1})^\f),\quad p_2=\Psi(\om 1(0^{m-1}1)^\f);\\
&\Psi(\om 0(10^{m-1})^\f)\le p_4\le \Psi(\om(0^{m-1}1)^\f),\quad p_3=\Psi(\om 0(1^{m-1}0)^\f).
\end{split}
\end{equation}
By the thickness as described in Lemmas \ref{l31}, in order to prove Theorem \ref{t41} (i) it suffices to prove {the inequalities}
\begin{equation}\label{43}
p_3-p_2\le p_2-p_1\qtq{and} p_3-p_2\le p_4-p_3
\end{equation}
for some large integer $m$. 

By \eqref{42} it follows that 
\begin{align*}
p_2-p_1&\ge \Psi(\om 1(0^{m-1}1)^\f)-\Psi(\om 1(01^{m-1})^\f)\\
&=\frac{1}{(\om 1(0^{m-1}1)^\f)_2}-\frac{1}{(\om 1(01^{m-1})^\f)_2}\ge\frac{(0^{n+2}10^\f)_2}{\big((\om 110^\f)_2\big)^2},
\end{align*}
\begin{align*}
p_4-p_3&\ge \Psi(\om0(10^{m-1})^\f)-\Psi(\om 0(1^{m-1}0)^\f)\\
&=\frac{1}{(\om 0(10^{m-1})^\f)_2}-\frac{1}{(\om 0(1^{m-1}0)^\f)_2}\ge \frac{(0^{n+2}10^\f)_2}{\big((\om 110^\f)_2\big)^2}
\end{align*}
and 
\begin{align*}
p_3-p_2&=\Psi(\om 0(1^{m-1}0)^\f)-\Psi(\om 1(0^{m-1}1)^\f)\\
&=\frac{1}{(\om 0(1^{m-1}0)^\f)_2}-\frac{1}{(\om 1(0^{m-1}1)^\f)_2}\le \;\frac{(0^{n+m}30^\f)_2}{\big((\om010^\f)_2\big)^2}.
\end{align*}
Take $m_0\ge 3$ such that 
\begin{equation}\label{44}
\frac{(0^{n+m}30^\f)_2}{(0^{n+2}10^\f)_2}<\frac{1}{2}\left(\frac{(\om010^\f)_2}{(\om 110^\f)_2}\right)^2
\end{equation}
for all $m\ge m_0$.
{Here the existence of $m_0$ follows from  that}  the left term of \eqref{44} tends to zero as $m\ra\f$, while the right term is a positive constant independent of $m$. 
Then \eqref{44} and the estimates of $p_2-p_1, p_4-p_3, p_3-p_2$ imply \eqref{43} for all $m\ge m_0$:
\begin{equation*} 
p_3-p_2\le\frac{(0^{n+m}30^\f)_2}{\big((\om010^\f)_2\big)^2}<\frac{(0^{n+2}10^\f)_2}{\big((\om 110^\f)_2\big)^2}\le\min\set{ p_2-p_1, p_4-p_3}.
\end{equation*}
Applying Lemma \ref{l31} we conclude that $\n_m+\la\n_m$ contains an interval for all $\la\ne 0$ and any $m\ge m_0$. 

Next, since $1\le p_1<p_2<p_3\le 2$,  we also infer from \eqref{44} and the estimates of $p_2-p_1, p_4-p_3, p_3-p_2$ for all $m\ge m_0$ the relations
\begin{align*}
\frac{p_3}{p_2}\le 1+\frac{(0^{n+m}30^\f)_2}{p_2\big((\om010^\f)_2\big)^2}
&<1+\frac{(0^{n+2}10^\f)_2}{p_1\big((\om 110^\f)_2\big)^2}\le \frac{p_2}{p_1}
\intertext{and}
\frac{p_3}{p_2}\le   1+\frac{(0^{n+m}30^\f)_2}{p_2\big((\om010^\f)_2\big)^2}
&<1+\frac{(0^{n+2}10^\f)_2}{p_3\big((\om 110^\f)_2\big)^2}\le\frac{p_4}{p_3}.
\end{align*}
{Applying Lemma \ref{l35} we conclude that the algebraic product $\n_m\cdot\n_m^\la $  contains an interval for all $\la\ne 0$ and any $m\ge m_0$. }

Since $\n_m\subset\n$ for all $m\ge 3$, this completes the proof.
\end{proof}

\section{Final remarks}
{The method used in the proofs of Theorems \ref{t11} and \ref{t41} can also be applied to many other Cantor sets that come up in dynamics. In this section we continue the investigation of the algebraic sum and product of $\u(x)$ for $x=1$.}
Recall  that $\u(1)$ is the set of univoque bases $q\in(1,2]$ such that $1$ has a unique $q$-expansion.  
As it is customory, let us simply write $\u$ instead of $\u(1)$.

Since both $\u +\u$ and  $\u \cdot \u$ contain an interval by Theorem \ref{t11}, it is natural to ask whether $\u +\u$ and $\u \cdot \u$ themselves are intervals.  
The answer is negative:

\begin{proposition}
\label{p51}
Neither $\u +\u$, nor $\u \cdot \u$ is an interval.
The same conclusion holds if we replace $\u$ by its topological closure $\overline{\u}$.
\end{proposition}

Before proving Proposition \ref{p51} we recall some results from {\cite{DeVriesKomornik2009, deVries-Komornik-Loreti-16, Kom-Lor-1998, Komornik_Loreti_2007}}   on the topological properties  of $\u$.
First, $\overline{\u}$ is a Cantor set and  $q_{KL}\approx 1.78723$ is its smallest element.
Next, we have
\begin{equation*}
\overline{\u}=[q_{KL}, 2]\setminus\bigcup(q_L, q_R),
\end{equation*}
where on the right-hand side we have a union {of} countably many pairwise disjoint open intervals: the connected components of $[q_{KL},2]\setminus\overline{\u}$.

Furthermore, for each of these intervals $(q_L, q_R)$ there exists a word $a_1\cdots a_m$ with $a_m=0$, satisfying the lexicographic inequalities
\begin{equation}\label{51}
(\overline{a_1\cdots a_m})^\f \prec \si^i((a_{1}\cdots a_m)^\f)\lle (a_1\cdots a_m)^\f \qtq{for all}i\ge 0
\end{equation}
and the equalities
\begin{equation}\label{52}
\Phi_1(q_L)=(a_1\cdots a_m)^\f\qtq{and} \Phi_1(q_R)=a_1\cdots a_m^+\overline{a_1\cdots a_m}\overline{a_1\cdots a_m^+} a_1\cdots a_m^+\cdots.
\end{equation}
Here $\si$ denotes the usual {left-}shift operator, and we use the notations 
\begin{equation*}
\overline{a_1\cdots a_m}:=(1-a_1)\cdots (1-a_m),\quad a_1\cdots a_m^+:=a_1\cdots a_{m-1}(a_m+1).
\end{equation*}
We recall that the left endpoints $q_L$ are algebraic integers, while the right endpoints $q_R$, called \emph{de Vries-Komornik numbers} in \cite{Kong_Li_2015}, are transcendental and their  expansions  $\Phi_1(q_R)$  are Thue-Morse type sequences. 

We also need an elementary lemma:

\begin{lemma}\label{l52}
Let  $A$ be a non-empty set of real numbers, and set 
\begin{equation*}
a:=\inf A,\quad b:=\sup A.
\end{equation*} 
If there exists a non-empty subinterval $(c,d)$ of $(a,b)$ such that 
\begin{equation*}
A\cap (c,d)=\emptyset\qtq{and}d-c>c-a,
\end{equation*}
then $A+A$ is not an interval.
\end{lemma}

\begin{proof}
Since $A+A$ meets both a neighborhood of both $2a$ and $2b$ by the definition of the infimum and supremum, it suffices to show that it does not meet the non-empty subinterval $(2c,a+d)$.

Let $x,y\in A$.
If $x\le c$ and $y\le c$, then $x+y\le 2c$.
Otherwise at least one of them is at least $d$. Since the other one is at least $a$, then $x+y\ge a+d$.
\end{proof}

\begin{proof}[Proof of Proposition \ref{p51}]
In order to prove that $\u +\u$ is not an interval, by the preceding lemma it suffices to find a connected component $(q_L,q_R)$ of $[q_{KL},2]\setminus\overline{\u}$ satisfying 
\begin{equation}\label{53}
q_R-q_L>q_L-q_{KL}.
\end{equation}

We claim that the interval $(q_L,q_R)$ associated with the word $a_1\cdots a_6=110100$ satisfies this inequality.

This word defines an interval $(q_L,q_R)$ indeed, because {it} satisfies the inequalities in  \eqref{51}:
\begin{equation*}
(001011)^\f\prec\si^i((110100)^\f)\lle (110100)^\f\qtq{for all}i\ge 0. 
\end{equation*}
In view of \eqref{52} the endpoints of $(q_L, q_R)$ satisfy the relations
\begin{equation*}
\Phi_1(q_L)=(110100)^\f\qtq{and}\Phi_1(q_R)=110101\,001011\;001010\,110101\;\cdots.
\end{equation*}
By a numerical calculation we have $q_L\approx 1.78854$ and $q_R\approx1.79656$.
Hence
\begin{equation*}
q_R-q_L>1.79654-1.78854=0.008
\end{equation*}
and
\begin{equation*}
q_L-q_{KL}\approx 1.78854-1.78723=0.00131,
\end{equation*}
so that the inequality \eqref{53} is satisfied.
The above proof remains valid for $\overline{\u}+\overline{\u}$ instead of $\u +\u$.

Next we consider the product $\u\cdot\u$.
Since it is homeomorphic to 
{\begin{equation*}
\ln\u+\ln\u=\set{\ln p+\ln q: p, q\in\u},
\end{equation*}}
it suffices to find a connected component $(q_L,q_R)$ of $[q_{KL},2]\setminus\overline{\u}$ satisfying 
\begin{equation}\label{54}
\ln q_R-\ln q_L>\ln q_L-\ln q_{KL},\qtq{i.e.,}\frac{q_R}{q_L}>\frac{q_L}{q_{KL}}.
\end{equation}
This is satisfied with the same interval $(q_L, q_R)\approx(1.78854, 1.79656)$ as in the first part of the proof because 
\begin{equation*}
\frac{q_R}{q_L}\approx 1.00448>  1.00073 \approx \frac{q_L}{q_{KL}}
\end{equation*}
by a numerical computation.
The  proof remains valid for $\overline{\u}\cdot\overline{\u}$ instead of $\u\cdot\u$.
\end{proof}

We end our paper with the following

\begin{conjecture}\label{conjecture53}
Both the algebraic {difference} $\u -\u$  and {quotient} ${\u}\cdot{\u}^{-1}$  are intervals. The same conclusion holds if we replace $\u$ by its topological closure $\overline{\u}$.
\end{conjecture}

\section*{Acknowledgements}
{The authors thank the anonymous referee for improving the presentation of the paper.} The second author was sponsored by NWO visitor’s travel grant 040.11.599.
The third author was supported by   NSFC No.~11401516. The fourth author was supported by NSFC No.~11671147, 11571144 and Science and Technology Commission of Shanghai Municipality (STCSM)  No.~13dz2260400.


\end{document}